\documentclass[12pt]{article}

\usepackage{amsfonts, amssymb}
\usepackage{amsmath}
\usepackage{verbatim}
\usepackage{indentfirst}
\usepackage{color}

\newtheorem{theorem}{Theorem}[section]
\newtheorem{proposition}[theorem]{Proposition}
\newtheorem{lemma}[theorem]{Lemma}
\newtheorem{corollary}[theorem]{Corollary}
\newtheorem{prop-def}{Proposition-Definition}[section]

\newtheorem{defi}[theorem]{Definition}

\newenvironment{proof}{\trivlist \item[\hskip \labelsep{\it Proof.}]}{
 \endtrivlist}

\begin{document}

\title{ Cross-section Lattices of ${\mathcal J}$-irreducible Monoids and Orbit Structures of Weight Polytopes
\thanks{Project supported by national NSF of China (No 11171202).}
\author{Zhuo Li, ~~You'an Cao, ~~Zhenheng Li}
}
\date{}
\maketitle

\vskip -35mm

\begin{abstract}
Let $\lambda$ be a dominant weight of a finite dimensional simple
Lie algebra and $W$ the Weyl group. The convex hull of $W\lambda$
is defined as the weight polytope of $\lambda$. We provide a new
proof that there is a natural bijection between the set of orbits
of the nonempty faces of the weight polytope under the action of the Weyl
group and the set of the connected subdiagrams of the extended
Dynkin diagram that contain the extended node $\{-\lambda\}$. We
show that each face of the polytope can be transformed to a
standard parabolic face. We also show that a standard parabolic
face is the convex hull of the orbit of a parabolic subgroup of
$W$ acting on the dominant weight. In addition, we find that the
linear space spanned by a face is in fact spanned by roots.
\medskip


\medskip

\end{abstract}

\def\J{{\cal J}}
\def\U{{\cal U}}
\def\Z {\mbox{{\bf Z}}}
\def\Q {\mbox{{\bf Q}}}
\def\U{{\cal U}}
\def\O{{\cal O}_{\Phi}}
\def\N{{\mbox{\bf N}}}
\def\F{{\mbox{\bf F}}}
\def\C{{\mbox{\bf C}}}
\def\R{{\mbox{\bf R}}}
\def\Z{{\mbox{\bf Z}}}
\def\a{\alpha}
\def\w{\omega}
\def\vw{\tilde{\omega}}
\def\b{\beta}
\def\g{\gamma}
\def\d{\delta}
\def\cI{\Pi_{\overline I^0}}
\def\I0{\overline I^0}
\def\L{\Lambda}
\def\la{\langle}
\def\ra{\rangle}

\baselineskip 15pt
\parskip 2pt

\section{Introduction}

The interests of studying the set of orbits of the faces of the
weight polytopes under the action of the Weyl groups arise partially from
the linear algebraic monoid theory \cite{Pu88, Re05, So95}.  A linear algebraic monoid is {\em irreducible} if it is irreducible as a variety. An irreducible monoid is {\em reductive} if its unit group is reductive. Let $M$ be a reductive monoid with unit group $G$. Consider the group
action of $G\times G$ on $M$ by $(g, h)\cdot a = gah^{-1}$ for $g,
h\in G$ and $a\in M$. The set $G\backslash M/G$ of orbits $GaG$
for this action is a poset with respect to inclusion. Associated with a reductive monoid is a polytope on which the Weyl group $W$ acts naturally. The lattice of orbits of $M$ is isomorphic to the lattice of  orbits of the faces of the polytope under the action of $W$. An
interesting problem in linear algebraic monoid theory is to find
the orbits of the reductive monoid $M$. In general, the problem is still open.

Let $B$ be a Borel subgroup of $G$ with maximal torus $T\subseteq B$. Denote by $\overline T$ the Zariski closure of $T$ in $M$, and let $E(\overline T) = \{e\in \overline T \mid
e^2=e\}$. The {\em cross-section lattice} of $M$ by definition is
\[
        \Lambda =\{e\in E(\overline T) \mid Be = eBe \},
\]
which is isomorphic to $G\backslash M/G$. In particular,
\[
    M = \bigsqcup_{e\in \Lambda} G e G \qquad\text{(disjoint union).}
\]

It is interesting and challenging to find the cross-section lattice
combinatorially. If $M$ has zero and its cross-section lattice has a unique minimal
nonzero element, then $M$ is referred to as {\em $\mathcal
J$-irreducible}. The cross-section lattices of $\mathcal
J$-irreducible monoids are described explicitly in \cite{PR}.
 In particular, let $G_0$ be a simisimple algebraic group and $\rho: G_0\to GL(V)$
 be a complex irreducible rational representation with highest weight $\lambda$. Then $G=K^*\rho(G_0)$ is reductive and
\[
    M(\lambda) = \overline {G}   \quad\quad\text{\rm (Zariski closure in End(V))}
\]
is a $\J$-irreducible monoid and all $\J$-irreducible monoids can
be constructed this way up to a finite morphism. Let $\mathfrak{ g}$ be the Lie algebra
of $G$ and $\Phi$ be its root system with a root basis
$\Pi=\{\a_1, \dots, \a_n\}$. Denote by $E$ the Euclidean space
spanned by $\Pi$, and by $\{\mu_1, \dots, \mu_n\}$ the
fundamental dominant weights with respect to $\Pi$. If $\lambda =
a_1\mu_1 + \cdots + a_n\mu_n$, the set
\[
    J = \{j\in {\bf {n}} \mid a_j = 0\}
\]
is called the {\em type of $M(\lambda)$} or simply the {\em type of $\lambda$} where ${\bf {n}}=\{1, \cdots, n\}$.
The result below describing the cross-section lattices of $\J$-irreducible monoids is a summary of Corollary 4.11 and Theorem 4.16 of \cite{PR}.

\vskip 2mm {\it {\bf Theorem A.} Let $J$ be the type of $M(\lambda)$ with $\lambda$ dominant. Then the nonzero elements of the cross-section
lattice of $M(\lambda)$ are in bijection with the subsets $I$ of
$\Pi$ where no connected component of $I$ lies entirely in $J$. } \vskip
2mm

The cross-section lattice of $M(\lambda)$ where the
dominant weight $\lambda$ is the highest root  of $\mathfrak{ g}$
was explicitly described in \cite[Theorem 3.1]{LC1} using
connected subsets of the affine Dynkin diagram.

The cross-section lattice of the $\J$-irreducible monoid $M(\lambda)$ can be characterized
using polytopes \cite{Pu88, Re05, So95}. The polytope associated with $M(\lambda)$ is the convex hull of
the orbit $W\lambda$, called the {\em weight polytope of $M(\lambda)$}, or simply the {\em weight polytope of $\lambda$}. If $\lambda$ is
 the highest root of $\mathfrak{ g}$ then the weight polytope is a {\em root polytope}.
 The following result can be deduced from \cite[Theorem 2]{KKMS73}
 and \cite[Theorem 8.4]{Pu88} and \cite[Section 5]{So95}.

\vskip 2mm
{\it {\bf Theorem B.} There is a one to one correspondence between
the cross-section lattice of $M(\lambda)$ and the set of orbits of
the faces of the weight polytope under the action of the Weyl
group. }
\vskip 2mm

The weight polytope of $\lambda$ depends only on the type $J$ of $\lambda$,
so we call $J$ the type of the weight polytope. All the strongly
dominant weights are of the same type $J=\emptyset$, the empty set.

\vskip 2mm {\it {\bf Theorem C.} If the type of the weight
polytope is $J$, then the set of orbits of the nonempty faces of the weight
polytope is in bijection with the set of subsets $I$ of $\Pi$
where no connected component of $I$ lies entirely in $J$. } \vskip
2mm

Though the theory of algebraic monoids accelerated in the last three decades,
some researchers do not know that the subject exists, and based on some minimal
evidence, some are discouraged by prerequisites which seem a little hard
to defeat.

In this paper, we restate Theorem C in terms of the language of Lie algebras
 and prove it using techniques from the theory of Lie algebras to
 make it more accessible to Lie algebraists. Adding a new node $\{-\lambda\}$
 to the Dynkin diagram of the Lie algebra $\mathfrak{ g}$ by drawing a single
 laced edge between $\{-\lambda\}$ and $\a_i$ if $(\lambda, \a_i)>0$ where $(~, ~)$ is the inner product on the Euclidean space $E$, we obtain
 an {\em extended Dynkin diagram}. Theorem C can then be described in Theorem D, which is a consequence of our Theorems \ref{bijection} and \ref{polyproperties}.

\vskip 2mm {\it {\bf Theorem D.} The orbits of the nonempty faces of the weight polytope
under the action of the Weyl group are in bijection with the
connected subdiagrams of the extended Dynkin diagram that contain
the extended node $\{-\lambda\}$.}
\vspace{2mm}

If $\lambda$ is the highest root of the Lie algebra $\mathfrak{
g}$, then Theorem D specializes the bijection for root polytopes proved in
\cite[Theorem 5.9]{CM}, which makes use of the properties of the root system of $\mathfrak{
g}$. The methods in
\cite{CM} motivate us to prove Theorem D without using
the theory of linear algebraic monoids. Our arguments are based on the weight set of $\mathfrak{
g}$, which is lacking of some advantages of the root system, so we need new techniques to over come this difficulty. This is reflected in the proofs of the results in Section \ref{proof}. We notice that different researchers study algebraic, geometrical, and combinatorial properties of root polytopes, and we refer the reader to \cite{ABH, CM2, M1, M2} for more details.

Our first main result Theorem \ref{keytheorem} states that a standard
parabolic face is the convex hull of the orbit of a parabolic
subgroup of $W$ acting on the dominant weight. This leads to the
proof of Theorem \ref{bijection}. The second main result Theorem
\ref{span} is an analogue of \cite[Corollary 5.4]{CM} for weight polytopes, stating that the
linear space spanned by a face of a weight polytope is in fact
spanned by roots. This result is used to prove that all nonempty faces of a
weight polytope are parabolic in Theorem \ref{polyproperties}.

\newpage
The rest of the paper is organized as follows. Section 2 provides some necessary
background knowledge. Section 3 is a summary of notation and terminology.
Section 4 is dedicated to the proof of the main results of the paper.

\section{Preliminaries}

Let $\mathfrak{ g}$ be a finite-dimensional complex simple Lie algebra. We fix a Cartan subalgebra $\mathfrak{ h}$ of
$\mathfrak{g}$ and assume that $\Phi$ is the root system of $\mathfrak{ g}$ with respect to $\mathfrak{ h}$.  Then $\mathfrak{g}$ has
the following root space decomposition
\[
\mathfrak{ g} =\mathfrak{ h}\oplus\sum_{\alpha\in \Phi} \mathfrak{
g}_{\alpha}.
\]
Let ${\bf {n}}=\{1,2,\ldots,n\}$, $\Pi=\{\a_i \mid
i\in{\bf {n}}\}$ be a root basis of $\Phi$, and $W$ be the
Weyl group generated by the simple reflections $\{s_{\a_i}\mid i\in
{\bf {n}}\}$. Let $\{\w_i \mid i\in {\bf {n}}\}$ be the fundamental weights with
respect to $\Pi$, and let $\vw_i=\frac{2\w_i}{(\a_i,
\a_i)}$ be the coweights of $\w_i$. We have $(\vw_i, \a_j)=\delta_{ij}$ for $i,j\in {\bf {n}}$.

Denote by
$P(\lambda)$ the set of weights of the irreducible highest weight module with highest dominant weight $\lambda$. Then $P(\lambda)$ is saturated, and there is a partial order on $P(\lambda)$: $\mu\le \nu$ if $\nu-\mu$ is a sum of simple roots. We will use repeatedly the following known results in the representation theory of Lie algebras.

\begin{proposition}\label{string}{\em (\cite[Proposition 3.6]{Kac})}
For $\mu\in P(\lambda)$ and $\a_i\in\Pi$, suppose that the $\a_i$-string through $\mu$ is: $\mu-p\a_i, \ldots, \mu, \ldots,\mu+q\a_i$. Then $p-q=\frac{2(\mu,\a_i)}{(\a_i,\a_i)}$.
In particular, if $(\mu,\a_i)>0$, then $\mu-\a_i\in P(\lambda)$.
\end{proposition}

\begin{proposition} \label{bki} {\em (\cite[Chapter VI, Proposition 24]{BBK})} Let $\Phi$ be a
root system and let $\Phi'$ be the intersection of $\Phi$ with a
subspace of $E$. Then $\Phi'$ is a root system in the subspace
spanned by $\Phi'$.
\end{proposition}

Let $I\subseteq {\bf {n}}$ and $W_I$ be the parabolic subgroup of $W$
generated by $s_{\a_i}$ for $i\in I$. Let $W^I$ be the minimal length
representatives of left cosets of $W_I$ in $W$. Then
\begin{eqnarray*}
W^I &=&\{w\in W \mid w{\color{blue}\cdot}\a_i >0 \mbox{ for } i\in I\}\\
&=&\{w\in W \mid l(ws_{\a_i})=l(w)+1\mbox{ for } i\in I\}.
\end{eqnarray*}

\begin{proposition}\label{decomposition}{\em(\cite[Chapter 5, Lemma C]{RK})} Every element
$w\in W$ has a unique decomposition $w=w^Iw_I$ such that
$l(w)=l(w^I)+l(w_I)$, where $w^I\in W^I, w_I\in W_I$.
\end{proposition}

\section{Notation and Terminology}

Fix a dominant weight $\lambda$ in the euclidean space $E$ spanned by $\Pi$, once and forever. The convex hull $\bold P$ of $P(\lambda)$ is a polytope equal to the convex hull of $W\lambda$.

\begin{defi} We call $\bold P$ the {\em weight polytope} of $\lambda$.
\end{defi}

Let $\lambda=\sum_{i=1}^n m_i\alpha_i$. Then all the
coefficients $m_i$'s are non-negative rational numbers. Without loss of generality, we
may assume that all $m_i$'s are non-negative integers since a dilation of a polytope and itself have
the same face lattice structure. The following is a consequence of Proposition 11.3 of \cite{Kac}.

\begin{proposition} If all the coefficients in $\lambda=\sum_{i=1}^n m_i\alpha_i$ are integers, then $P(\lambda)$ is the
intersection of the weight polytope $\bold P$ with the root lattice.
\end{proposition}

If we express a weight $\mu$ as a linear combination
of simple roots $\a_i$, we use $c_i(\mu)$ to denote the coefficient of $\a_i$. For example, $c_i(\lambda) = m_i$.
Let $P_i$ be the set of weights $\mu$ such that $(\mu, \vw_i)=m_i$, in other words,
\[
    P_i = \{\mu\in P(\lambda) \mid c_i(\mu) = m_i \}.
\]
\begin{defi}The convex hull of
$P_i$ is called  the {\rm $i$th coordinate face} of $\bf P$, and will be denoted by $F_i$.
\end{defi}

Each coordinate face $F_i$ is a face of the weight
polytope $\bold P$ since $F_i$ sits in the hyperplane $\{x\in E
\mid (x, \vw_i)=m_i\}$ and all other weights in $F_i$ are in the half-space
$\{x\in E \mid (x, \vw_i)\le m_i\}$. It is, however, possible that $F_i$ is not a facet.

\begin{defi} \label{parabolicFace}
A face of $\bold P$ is {\em standard
parabolic} if it is an intersection of coordinate faces; a
face of $\bold P$ is {\em parabolic} if it can be
transformed into a standard parabolic face by an element of $W$.
\end{defi}

For any $I\subseteq {\bf {n}}$, the face
\[
    F_I=\{\mu\in \bold P \mid c_i(\mu)=m_i~{\rm for}~i\in I\}
\]
is standard parabolic, since $F_I=\cap_{i\in I}F_i$. On the other hand, for any standard
parabolic face $F$, there exists $I\subseteq {\bf {n}}$ such that $F=F_I$.
Clearly,
\[
P_I=\cap_{i\in I}P_i=\{\mu\in P(\lambda) \mid c_i(\mu)=m_i~{\rm for}~i\in I\}
\]
is the set of all weights in $F_I$. Note that $F_{{\bf {n}}} = P_{{\bf {n}}}= \{\lambda\}$. For convenience, let $F_\emptyset = {\bf P}$ and  $P_\emptyset = P(\lambda)$. So the standard parabolic face $F_I$ is not empty for all $I\subseteq {\bf {n}}$.

\begin{defi}\label{extendedDiagram}
An extended Dynkin diagram is the diagram obtained by adding a
new node $\{-\lambda\}$ to the Dynkin diagram, in which a
single laced edge is drawn between $\{-\lambda\}$ and $\a_i$ if and only if
$(\lambda, \a_i)>0$ for $i\in {\bf n}$.
\end{defi}

\section{Weight Polytope}\label{proof}
The purpose of this section is to show in Theorem \ref{bijection} that there is a bijection between the set of standard
parabolic faces and the set of certain connected components in the
extended Dynkin diagram, and then in Theorem \ref{polyproperties} that all nonempty faces of a weight polytope are parabolic.

If $I$ is a subset of ${\bf {n}}$, let $\overline I$ be the complement
set of $I \subseteq {\bf {n}}$, and let $\Pi_I = \{\a_i \in \Pi \mid i\in
I\}$. We begin our discussion by introducing the following two lemmas.
\begin{lemma}\label{chainlemma}
Let $\mu, \nu$ be weights in a standard parabolic proper face $F_I$
of $\bold P$. If $\mu<\nu$ then there are simple roots
$\a_{i_1},\ldots,\a_{i_k}\in \Pi_{\bar I}$ such that $\nu=\mu+\a_{i_1}+\cdots+\a_{i_k}$ and all
$\mu+\a_{i_1}+\cdots+\a_{i_s}$ are weights in $F_I$ for $s=1,2,\ldots,k$.
\end{lemma}

\begin{proof}  Let $\nu-\mu=\sum_{j\in J} a_j\a_j$ where
$J\subseteq {\bf {n}}$ and $a_j$ are positive integers for all $j\in J$. We claim that $I \cap J = \emptyset$  and $(\nu-\mu, \a_{j_0})>0$ for some $j_0\in J$. Otherwise, if $i\in I\cap J$, then $c_i(\mu) =m_i-a_i<m_i$, which contradicts that $\mu\in F_I$. If $(\nu-\mu, \a_j)\le 0$ for all $j\in J$, then $(\nu-\mu, \nu-\mu)=0$, and hence $\nu=\mu$, which is a contradiction.

Use induction on the
height of $\nu-\mu$:
$
{\rm ht}(\nu-\mu)=\sum_{j\in J}a_j.
$
If ${\rm ht}(\nu-\mu)=1$, then the result is true.
Assume the result in Lemma \ref{chainlemma} is true for ${\rm ht}(\nu-\mu)=k-1$ where
$k>1$. We show that the result is true for ${\rm
ht}(\nu-\mu)=k$ case by case.

{\bf Case 1}: $(\nu, \a_{j_0})>0$.

It follows from Proposition \ref{string} that
$\nu-\a_{j_0}$ is a weight in $F_I$, and $\mu<\nu-\a_{j_0}$ with ${\rm
ht}(\nu-\a_{j_0}-\mu)=k-1>0$. Applying the induction hypothesis, we
have $\nu-\a_{j_0}=\mu+\a_{i_1}+\cdots+\a_{i_{k-1}}$ such that all
$\mu+\a_{i_1}+\cdots+\a_{i_s}$ are weights in $F_I$ for $s=1,2,\ldots,k-1$.
Let $\a_{i_{k}}=\a_{j_0}$. Then $\nu=\mu+\a_{i_1}+\cdots+\a_{i_{k-1}}+\a_{i_k}$ and
all $\mu+\a_{i_1}+\cdots+\a_{i_s}$ are weights in $F_I$ for
$s=1,2,\ldots,k$.

{\bf Case 2}: $(\nu, \a_{j_0})\le 0$.

In this case, $(\mu, \a_{j_0})<(\nu, \a_{j_0})\le 0$. Then
$\mu+\a_{j_0}$ is a weight in $F_I$ by Proposition \ref{string}.
Therefore, ${\rm ht}(\nu-(\mu+\a_{j_0}))=k-1>0$ and $\mu+\a_{j_0}<\nu$.
The induction hypothesis shows that
$\nu=(\mu+\a_{j_0})+\a_{i_1}+\cdots+\a_{i_{k-1}}$ where all
$(\mu+\a_{j_0})+\a_{i_1}+\cdots+\a_{i_s}$ are weights in $F_I$ for
$s=1,2,\ldots,k-1$. Reorder the indices by setting $\a_{i_1}=\a_{j_0}$
and $\a_{i_t}=\a_{i_{t-1}}$ for $t=2, \dots, k$. Then we have
$\nu=\mu+\a_{i_1}+\cdots+\a_{i_k}$ and all $\mu+\a_{i_1}+\cdots+\a_{i_s}$ are
weights in $F_I$ for $s=1,2,\ldots,k$. \hfill $\Box$

\end{proof}

The following lemma is easy to understand by intuition, however, its proof is not trivial and the authors don't find it in any
existing literature.

\begin{lemma} \label{descent}Assume that $w\lambda=
\lambda-\sum_{j\in J}a_j\a_j$ where $w\in W$, $a_j$ are positive integers, and
$\a_j\in \Pi$ for all $j\in J$. Then there exists $u\in W_J$ such
that $w\lambda=u\lambda$.
\end{lemma}

\begin{proof}
By Proposition \ref{decomposition}, $w =w^J w_J$ where
$w^J\in W^J$ and $w_J\in W_J$ with $l(w) =l(w^J)+l(w_J)$. We apply
induction on the length $l(w)$.

If $l(w)=0$ then result is trivial. If $l(w)=1$ then $w=s_{\a_i}$ for some $\a_i\in \Pi$. If
$i\in J$, then we are done. If $i\notin J$, then $\a_i$ is linearly independent of $\{\a_j \mid j\in
J\}$. But
\[
w\lambda = s_{\a_i}\lambda =\lambda-\frac{2(\lambda,
\a_i)}{(\a_i,\a_i)}\a_i =\lambda-\sum_{j\in J}a_j\a_j.
\]
So, $J$ must be empty, and then $W_J = 1$. Therefore, $w\lambda=\lambda$.

Assume that $l(w)>1$. Let $w^J=s_{\a_1}\dots s_{\a_r}$ be a reduced
expression. If $w^J=1$ then we are done. If $w^J\not=1$ then $r\ge
1$. Then $w^J\a_r<0$ by \cite[Section 4.3, Theorem B]{RK} and
$r\notin J$. Write $w_J\lambda =\lambda-\sum_{i\in J}b_i\a_i$
where $b_i\ge 0$ for all $i\in J$. We have
\begin{eqnarray*}
 (w\lambda,\  w^J\a_r)&=&(w^J w_J\lambda,\ w^J\a_r)=(w_J\lambda,\
 \a_r)\\
 &=&(\lambda-\sum_{i\in J}b_i\a_i,\ \a_r)\ge 0{\color{blue},}
 \end{eqnarray*}
 since $r\notin J$ and $(\a_i, \a_r)\le 0$ for $i\in J$.
 Divide the discussion into two cases.

{\bf Case 1}: $(w\lambda,\ w^J\a_r)=0$. We have
\[(\lambda,w_J^{-1}\a_r)=(w_J\lambda, \a_r)=0.\]
It follows that
\[
s_{w_J^{-1}\a_r}\lambda = w_J^{-1}s_{\a_r}w_J\lambda =\lambda,
\]
and so $s_{\a_r}w_J\lambda =w_J\lambda$. Hence,
\[
    w\lambda=w^J w_J\lambda = s_{\a_1}\dots s_{\a_{r-1}}w_J\lambda.
\]
But $l(s_{\a_1}\dots
s_{\a_{r-1}}w_J)<l(w)$. From the induction hypothesis, there exists an element $u\in
W_J$ such that $s_{\a_1}\dots s_{\a_{r-1}}w_J\lambda=u\lambda$.
Thus, $w\lambda=u\lambda$.

{\bf Case 2}: $(w\lambda,\ w^J\a_r)>0$. Now, $(\lambda-\sum_{j\in J}a_j\a_j,\ w^J\a_r)> 0$. Then,
\[\mu=\lambda-\sum_{j\in J}a_j\a_j-\ w^J\a_r\] is
a weight by Proposition \ref{string}.  Since $\mu\le\lambda$ and
$w^J\a_r$ is a negative root, $w^J\a_r$ has to be linear
combination of $\a_j$'s for $j\in J$. Write $\gamma = w^J\a_r$.
Then $w^Js_{\a_r}=s_{\gamma}w^J$ and
\[
s_{\gamma}w\lambda = w^Js_{\a_r}w_J\lambda =\lambda-\sum_{j\in
J}c_j\a_j\]
 where $c_j\ge 0$. Since $l(w^Js_{\a_r}w_J)<l(w)$, by
the induction hypothesis, there exists $v\in W_J$ such that
$w^Js_{\a_r}w_J\lambda =s_{\gamma}w\lambda = v\lambda$. Let
$u=s_{\gamma}v\in W_J$. Therefore, $w\lambda =
s_{\gamma}v\lambda=u\lambda$. This completes the proof. \hfill $\Box$

\end{proof}

For a subset $\Sigma$ of $\{-\lambda\}\cup\Pi$, denote by
$\Gamma(\Sigma)$ the subdiagram of the extended Dynkin diagram
having $\Sigma$ as vertices, and $\Gamma^0(\Sigma)$ the connected
component of $\Gamma(\Sigma)$ containing $\{-\lambda\}$. Denote by
$V_I$ the set of vertices of $\Gamma^0(\Pi_{\overline I} \cup
\{-\lambda\})$. Let $\cI = \Pi \cap V_I$ and let $\I0$ be the set
of indices of the vertices in $\cI$. Let $E_{F_I}$ be the space spanned by $\{\mu - \nu \mid \mu, \nu\in F_I\}$.

\begin{theorem} \label{keytheorem} Let $I\subseteq {\bf {n}}$ and $u_0$ be the longest
element in $W_{\overline I^0}$. If $F_I$ is a standard parabolic
face of $\bf P$, then

\begin{itemize}
\item[{\rm(1)}] $F_I$ is the convex hull of
$W_{\overline I^0}\lambda$,

\item[{\rm(2)}]  $u_0\lambda$ is the least
element in $P_I$, and

\item[{\rm(3)}] $E_{F_I}$ is spanned by the simple roots in $\Pi_{\I0}$
and $\dim E_{F_I} =|\I0|$.

\end{itemize}
\end{theorem}
\begin{proof}
For (1), let $w\in W_{\overline I^0}$. Then $w\lambda
=\lambda-\sum_{j\in{\I0}}a_j \a_j \in P_I$, since $\I0\cap I=\emptyset$.
Notice that $P_I$ is the set of weights in $F_I$. We obtain that $w\lambda\in F_I$, and the convex hull of $W_{\overline I^0}\lambda$ is contained
in $F_I$.

Conversely, for any $w\in W$ assume that $w\lambda
=\lambda-\sum_{j\in J}a_j\a_j$ with $a_j>0$. Then $J\subseteq
\overline I$, since $c_i(w\lambda) =c_i(\lambda)=m_i$ for all $i\in
I$. By Proposition \ref{descent}, there exists an element
$u\in W_J$ such that $u\lambda =w\lambda$. If $\Pi_J\cup
\{-\lambda\}$ is connected in the extended Dynkin diagram, then we
are done. Otherwise, we can assume that $J = J'\cup \I0 $. Notice that
$\cI\cup \{-\lambda\}$ is a connected component in the Dynkin
diagram. Then $W_J\cong W_{\overline I^0}\times
W_{J'} $. Write $u=xy$, where $x\in W_{\overline I^0}$
and $y\in W_{J'}$. Thus, $w\lambda =u\lambda=xy\lambda=x\lambda$
since $(\lambda, \a_j)=0$ for all $j\in J'$. In other words, $w\lambda\in W_{\overline I^0}\lambda$. It follows that $F_I$ is included in the convex hull of $W_{\overline I^0}\lambda$. This completes the proof of (1).

To prove (2),  it suffices to show that $u_0\lambda<w\lambda$ for
all $w\in W_{\overline I^0}$, since other weights in $P_I$ are
linear combinations of $W_{\overline I^0}\lambda$ with non-negative
coefficients whose sum is equal to $1$. Thanks to $\lambda\ge u_0^{-1}w\lambda$, we have
$\lambda-u_0^{-1}w\lambda=\sum_{j\in {\I0}}a_j\a_j$ with $a_j\ge 0$.
 It follows that
\[
\begin{aligned}
    u_0\lambda- w\lambda    &= u_0(\lambda-u_0^{-1}w\lambda) \\
                                &=  -\sum_{j\in {\I0}}a_{j}(-u_0(\a_{j}))
\end{aligned}
\]
where $-u_0(\a_j)\in \cI$ since $u_0\cI
= -\cI$. So $u_0\lambda\le w\lambda$.
This proves (2).

We now prove (3). If $u_0\lambda=\lambda$, then $\lambda$ is the least element of
$W_{\overline I^0}\lambda$ by (2) and $W_{\overline I^0}\lambda=\{\lambda\}$. It follows that $\cI=\emptyset$ and $F_I = \{\lambda\}$.
If $u_0\lambda<\lambda$, let
\begin{eqnarray}\label{Li1}
\lambda-u_0\lambda=\sum_{j\in \I0} a_j\a_j .
\end{eqnarray}
By Lemma \ref{chainlemma}, there are simple roots $\a_{i_1}, \dots,
\a_{i_k}$ such that $u_0\lambda+\a_{i_1}+\dots+\a_{i_s}$ are weights in $F_I$
for $s=1, \dots, k$. Any element $\lambda-w\lambda$ is a linear combination of $\a_{i_1}, \dots, \a_{i_k}$, for $w\in
W_{\overline I^0}\lambda$. Hence $\a_{i_1}, \dots, \a_{i_k}$ span
$E_{F_I}$. But $\a_{i_s}\in\cI$ by (1). The dimension of $F_I$
is the number of nonzero terms on the right-hand side of
(\ref{Li1}).
We claim that $a_j>0$ for all $j\in \I0$. To
this end, for any $\a_j\in \cI$, we can find a shortest connected
path: \,$\a_j,\a_{j_1}, \dots, \a_{j_m}, \{-\lambda\}$, \,from $\a_j$ to $\{-\lambda\}$ in the extended Dynkin diagram.
Then the coefficient of $\a_j$ in $\lambda -s_{\a_j}s_{\a_{j_1}}\dots s_{\a_{j_m}}\lambda$  is positive.  Hence $a_j>0$  in (\ref{Li1}) for $j\in\I0$, since $s_{\a_j}s_{\a_{j_1}}\dots s_{\a_{j_m}}\lambda\ge u_0\lambda$. This completes the proof. \hfill $\Box$
\end{proof}

Note that for different subsets $I, J\subseteq {\bf {n}}$, it is possible
that  $F_I=F_J$.
\begin{corollary}\label{facesEqual}For any two subsets $I, J\subseteq {\bf {n}}$, $F_I=F_J$
if and only if $\cI = \Pi_{\overline J^0}$.
\end{corollary}

\begin{proof} $(\Rightarrow)$ Use $\eta_I$ to denote the least element of $F_I$. If $F_I=F_J$, then they have
the same least element $\eta_I(=\eta_J)$.  By the proof of (3) in
Theorem \ref{keytheorem},
\[
\lambda-\eta_I=\sum_{i\in \I0}a_i\a_i\quad{\rm and}\quad
\lambda-\eta_J =\sum_{i\in \overline J^0}b_i\a_i,
\]
where $a_i>0$ for all $i\in \I0$, and $b_i>0$ for all $i\in
\overline J^0$.  It follows that $\cI = \Pi_{\overline J^0}$.

 $(\Leftarrow)$ This is trivial by (1) of Theorem \ref{keytheorem}.
 \hfill $\Box$
\end{proof}

The theorem below follows from Theorem
\ref{keytheorem} and Corollary \ref{facesEqual}.

\begin{theorem}\label{bijection}
The standard parabolic faces of $\bold P$
are in bijection with the connected subdiagrams of the extended
Dynkin diagram that contain $\{-\lambda\}$. This bijection is an
isomorphism of posets with respect to inclusion.
\end{theorem}

This theorem is the Lie algebra version of Theorem C, which is a consequence of \cite[Theorem 4.16]{PR} and \cite[Theorem 2]{KKMS73} and \cite[Theorem 8.4]{Pu88} and \cite[Section 5]{So95}. In fact, the bijection can be explicitly described as:
$
    F_I\to \I0.
$

\begin{corollary}
All coordinate faces are distinct.
\end{corollary}

\begin{proof}
Let  $\eta_i$ be the least element of $F_i$. We have
 $(\eta_i, \a_j)\le 0$ for all $j \not= i$. Otherwise
 $\eta_i-\a_j$ is a weight in
 $P_i\subseteq F_i$ by Proposition \ref{string}
and $\eta_i$ would not be minimal. This yields $(\eta_i, \a_i)> 0$
since $(\eta_i, \eta_i)>0$. Thus $\eta_i \neq \eta_j$, and
hence $P_i\not= P_j$ if $i\not= j$. The coordinate
faces $F_i$ and $F_j$ are then distinct since $P_i$ is exactly
the set of weights in $F_i$ for all $i\in {\bf {n}}$.
\hfill $\Box$

\end{proof}

The following is a criterion that tests whether a coordinate face is a facet.
\begin{corollary} Let $F_i$ be a coordinate face and $\eta_i$ be the
least element in $F_i$. Then $F_i$ is a facet if and only if
$(\eta_i, \vw_j)\not=m_j$ for all $j\not=i$.
\end{corollary}

\begin{proof} ($\Rightarrow$) By Theorem \ref{keytheorem}, if $F_i$ is a facet
then $\dim \,F_i=|\cI| =n-1$ where $I=\{i\}$. This implies that
$\cI=\Pi\setminus \{i\}$ and $\eta_i=\lambda-\sum_{\a_j\in\cI}
a_j\a_j$ with $a_j>0$ for all $j\not= i$. Therefore, $(\eta_i,
\vw_j)\not=m_j$ for all $j\not= i$.

($\Leftarrow$) If $(\eta_i, \vw_j)=m_j-a_j \not=m_j$ for all
$j\not=i$, then $a_j>0$ for all $j\not=i$. By Lemma
\ref{chainlemma} or the proof of Theorem \ref{keytheorem}, $E_{F_i}$ is
spanned by $\{\a_j \mid j\in {\bf {n}}\setminus\{i\}\}$. Hence $\dim F_i=n-1$. \hfill $\Box$
\end{proof}

\begin{lemma}\label{stringSum}
Let $\a$ be a root in $\Phi$ and $\mu$ be a weight in
$P(\lambda)$ and denote the $\a$-string through $\mu$ by
$\a-(\mu)$.  Then
\[
\sum_{\g\in \a-(\mu)} (\g, \a)=0.
\]
\end{lemma}

\begin{proof}
Without loss of generality, we assume that $\mu$ is
the start of its $\a$-string.  Then the weight string is
\[
\mu,\, \mu+\a,\, \ldots, \mu+q\a, \quad \mbox{where }~ q=-\frac{2(\mu,\a)}{(\a, \a)}.
\]
If $q$ is even, the middle weight $\mu-\frac{(\mu,\a)}{(\a, \a)}\a$ is
orthogonal to $\a$, and the sum of the two symmetric weights
about the middle weight is orthogonal to $\a$. If $q$ is odd, the
sum of the two symmetric weights is orthogonal to $\a$.
\hfill $\Box$.
\end{proof}

\begin{proposition}
The barycenter of the weights in the $i$th coordinate face $F_i$ is
parallel to the $i$th fundamental weight, and
\[
    \sum_{\mu\in P_i} \mu = \frac{m_i |P_i|}{(\vw_i, \vw_i)} \vw_i.
\]
\end{proposition}

\begin{proof}  For each
$j\in {\bf {n}}\setminus\{i\}$, $P_i$ is a disjoint union of
$\a_j$-strings. It follows from Lemma \ref{stringSum} that
\[
    \sum_{\mu\in P_i}(\mu, \a_j)=0.
\]
Hence, $\sum_{\mu\in P_i}\mu=a\,\vw_i$. In view of $(\mu, \vw_i)=m_i$ for
all $\mu\in P_i$, we have
\[
(a\,\vw_i,\vw_i)=\left(\sum_{\mu\in P_i}\mu \, , ~\vw_i\right)=
| P_i|(\mu,\vw_i)=m_i| P_i|.
\]
The desired result follows.
\hfill $\Box$
\end{proof}

\begin{lemma} Let $I\subseteq {\bf {n}}$.  Then the barycenter of the
standard parabolic face $F_I$ is $\sum_{i\in I} a_i \vw_i$ where
$a_i$ are non-negative rational numbers for all $i\in I$.
\end{lemma}

\begin{proof}
As $P_I$ is a disjoint union of $\a_j$-strings for any fixed $j\in
{\bf {n}}\setminus I$, we know that $(\sum_{\mu\in P_I}\mu \, , ~\a_j)=0$ for all $j\notin I$.
Thus, $\sum_{\mu\in P_I}\mu=\sum_{i\in I}a_i\vw_i$. For any $j\in
I$, we have
\[
\sum_{i\in I}a_i(\vw_i,\vw_j)=\left(\sum_{\mu\in
P_I}\mu \, , ~\vw_j\right)=| P_I|(\mu,\vw_j).
\]
For $i,j\in I$ and $\mu\in P_I$, the numbers $(\vw_i,\vw_j)$ and $(\mu,\vw_j)$ are rational numbers, so are all $a_i$'s. We claim that $(\mu, \a_i)\ge 0$ for $\mu\in P_I$. If not, it follows
from $(\mu, \a_i) < 0$ that $\mu+\a_i\in P(\lambda)$, which shows that
$c_i(\mu+\a_i)=m_i+1$. This contradicts that that $\mu \le \lambda$. Thus $a_i\ge 0$, as desired.
  \hfill $\Box$
\end{proof}

By the lemma above, the barycenter of a standard parabolic face is in
the closure of the fundamental Weyl chamber \cite{Hum}. Since the Weyl group acts
on the fundamental chamber transitively, we have
\begin{corollary} Two distinct standard parabolic faces of the weight polytope $\bf P$
cannot be transformed into one another by elements of the Weyl group.
\end{corollary}

\begin{lemma}\label{twobasic}
  Let $\bold P$ be a polytope (or cone). Let $\nu$ be a fixed
  vertex. Suppose $[\mu, \nu]$ is an edge (or spans an edge in the cone) with
\[
  \mu-\nu = a_1 (x_1-\nu)+\dots+ a_n(x_n-\nu)
\]
where $a_i>0$ and $x_i\in \bold P$. Then $x_i$ are in the same
edge that contains $[\mu, \nu]$.
\end{lemma}
\begin{proof} Assume that $[\mu,\nu]$ (or its span for the cone) is an intersection of
facet $F_k ~ (k=1,\dots, r)$.  Assume that the linear equation for
the hyperplane containing $F_k$ is $L_k(x)=0$ for $k\in [r]$. We
want to prove that $L_k(x_i)=L_k(\nu)$ for all $k\in [r]$ and
$i\in {\bf {n}}$. We have
\[
L_k(\mu-\nu) =
a_1(L_k(x_1)-L_k(\nu))+\dots+a_n(L_k(x_n)-L_k(\nu)) =0.\]

Since $x_i\in \bold P$, $L_k(x_i)-L_k(\nu)$ have the same sign
for all $x_i ~(i=1, \dots, n)$. Therefore, $L_k(x_i)=L_k(\nu)$
for all $k\in [r]$.
This means that $x_i$ is in the same edge that contains
$[\mu,\nu]$. \hfill $\Box$

\end{proof}

\begin{lemma}\label{edgeroot}
    Every edge of $\bold P$ is parallel to a root.
\end{lemma}
\begin{proof} It suffices to prove this for edges attached to
the highest dominant weight $\lambda$. Assume that $\mu=w\lambda \not=\lambda$ and the segment $[\lambda,
\mu]$ is an edge.  We assert that $\mu$ and $\lambda$ are
separated by one and only one reflection hyperplane $H_{\a}$ where
$\a$ is a positive root and $H_{\a}=\{x\in E \mid (x, \a)=0\}$.

If not, there are two positive distinct roots $\a, \b$  such that
$\mu$ and $\lambda$ are separated by two different reflection
hyperplanes $H_{\a}$ and $H_{\b}$. Then  $(\lambda, \a)>0$, $(\mu,
\a)<0$ and $(\lambda, \b)>0$, $(\mu, \b)<0$.  Therefore,
$\lambda-\a, \lambda-\b, \mu+\a$, and $\mu+\b$ are weights in
$P(\lambda)$. We have

\[
\mu-\lambda = \frac{(\mu+\a -\lambda) + (\mu+\b-\lambda) +
(\lambda-\a-\lambda)+ (\lambda-\b-\lambda)} {2},
\]
which is a contradiction by Lemma \ref{twobasic}. Thus $\mu $ and $\lambda$
are separated by one hyperplane $H_{\a}$. But $\mu $ and $\lambda$
are in the closures of two different Weyl chambers, one of which
is transformed to the other by the reflection $s_\a$. It follows
that $\mu = s_{\a} \lambda$
 from the fact that the Weyl group acts on all the chambers
 simply transitively. Here $\a$ is a positive root
 but not necessary a simple root. \hfill $\Box$
\end{proof}


The following theorem follows directly from the above lemma.

\begin{theorem} \label{span} Given a face $F$ of $\bold P$,
let $E_F = $ {\rm Span}$\{\mu-\nu \mid \mu, \nu\in F\}$.
Then the space $E_F$ is spanned by roots for any face $F$ of $\bold P$.
\end{theorem}

Now we are ready to show that all nonempty faces of a weight polytope are
parabolic. The proof of the following theorem is patented after that of
\cite[Theorem 5.6]{CM}. For completeness, we give the proof
here.

\begin{theorem}\label{polyproperties}
Every nonempty face of a weight polytope is parabolic.
\end{theorem}
\begin{proof} Let $F$ be a face with $\dim  F=n-p$ where $1\le p\le n$. We use induction on $p$ to prove the desired result. If $p=1$ then $F$ is a facet. From Theorem \ref{span} and Proposition
\ref{bki} it follows that $\Phi\cap E_F$ is a root subsystem of $\Phi$ with rank
$n-1$. Let $\Pi'$ be a root basis of $\Phi\cap
E_F$. Then there exists $w\in W$ such that $w\Pi'\subseteq \Pi$.  Let
$\a_i$ be the only fundamental root in $\Pi$ which does not belong
to $w\Pi'$. So $wE_F$ is perpendicular to $\vw_i$,
denoted by $wE_F\perp \vw_i$. Therefore, there exists a real number
$a$ such that $(x, \vw_i)=a$ for all $x\in wF$. It follows that
$\{x\in E\,\mid\,(x,\vw_i)=a\}$ is a supporting hyperplane of the weight polytope. Let $l_i=(w_0\lambda, \vw_i)$ where $w_0$ is the longest
element in the Weyl group $W$.  Then $l_i\le a\le m_i$. This forces
$a=m_i$ or $a=l_i$. If not, vertices $\lambda$ and $-w_0\lambda$
will be at different sides of the hyperplane $\{x\in
E\,\mid\,(x,\vw_i)=a\}$, a contradiction. If $a=m_i$ then $wF=F_i$. If
$a=l_i$, then $-l_i = (\lambda, \vw_{i'})=m_{i'}$ where
$\vw_{i'}=-w_0(\vw_i)$. Hence $w_0wF=F_{i'}$.

Now we assume that $n\ge p > 1$ and all faces of the weight polytope $\bold P$ of dimension
greater than $n-p$ are parabolic. By the induction hypothesis,
without loss of generality, we can assume that $F$ is a facet of a
standard parabolic face $F_I$ where $\dim  F=n-p$ and $\dim
F_I=n-p+1$. By Theorem \ref{keytheorem}, $F_I$ is the convex hull of
$W_{\I0}\lambda$ and $F_I$ is spanned by $\{\a_i \mid i\in \I0\}$. Note that
there exists $w\in W_{\overline I^0}$
such that $wE_F\perp \vw_j$ where $j\in \I0$. Meanwhile, there
exists a real number $a$ such that $(x, \vw_j)=a$ for all $x\in
wF$. Let $l_j=(\eta_I, \vw_j)$. Recall $\eta_I=u_0\lambda$ is the
least element in $F_I$ where $u_0$ is the longest element in $
W_{\overline I^0}$. Then $a=l_j$ or $a=m_j$.

If $a=m_j$ then $\lambda\in F$. Hence $wF =\{\lambda+\mu\,\mid\,\mu\in
E_{F_I}\cap \vw_j^{\perp}\}\cap F_I =F_I\cap F_j$. We are done.
If $a=l_j$, then $\eta_I\in wF$. Hence $wF
=\{\eta_I+\mu\,\mid\,\mu\in E_{F_I}\cap \vw_j^{\perp}\}\cap F_I$. Let $\pi$ be the orthogonal
projection of $E$ onto $E_{F_I}$. Then $E_{F_I}\cap \vw_k^{\perp}=
E_{F_I}\cap \pi(\vw_k)^{\perp} $ for all $k\in \I0$. Hence
$\{\pi(\vw_k) \mid k\in \I0\}$ is the coweights with respect to $\cI$
in root system $\Phi(\cI)$. Thus, there exists $j'$ such that
$u_0(\pi(\vw_j)) =-\pi(\vw_{j'})$. Since $u_0(F_I)=F_I$, we have
$u_0(E_{F_I}\cap \vw_j^{\perp})=E_{F_I}\cap \pi(\vw_{j'})^{\perp}
=E_{F_I}\cap \vw_{j'}^{\perp}$. Thus, $F$ is parabolic, since $u_0wF=
\{u_0(\eta_I)+\mu\mid\mu\in E_{F_I}\cap \vw_{j'}^{\perp}\} \cap F_I =\{
\lambda+\mu\mid\mu\in E_{F_I}\cap \vw_{j'}^{\perp}\} \cap F_I=F_I\cap
F_{j'}$. \hfill $\Box$
\end{proof}

We conclude this section by giving the
$f$-polynomial of weight polytopes.

\begin{theorem} Let ${\cal I}=\{I\subseteq {\bf n} \,\mid\, \Gamma(\Pi_I \cup \{-\lambda\}) \mbox{ is connected } \}$ and let
$J$ be the type of $\lambda$. For each $I\in{\cal I}$, let $I^*=I\cup \{j\in J \mid (\a_j, \a_i)=0 \mbox{ for all } i\in
I\}$. Then the f-polynomial of $\bold P$ is

\[
    \sum_{I\in {\cal I}} \frac{| W|}{|W_{I^*}|} t^{|I|}.
\]

\end{theorem}
\begin{proof} Consider the action of the Weyl group $W$ on the faces of $\bold P$.
A simple calculation yields that $W_{I^*}$ is the stabilizer of the standard
parabolic face $F$, where $F$ is the convex hull of $W_I\lambda$. \hfill $\Box$
\end{proof}

{\bf Acknowledgements}~ The authors thank Dr. Reginald Koo for his careful reading of and comments on the paper.

\newpage
\baselineskip 15pt
\vspace{1cm}
\noindent Zhuo Li \\
Department of Mathematics \\
Xiangtan University\\
Xiangtan, Hunan 411105, P. R. China\\
\noindent Email: zhuoli123@gmail.com \\

\noindent You'an Cao  \\
Department of Mathematics \\
Xiangtan University\\
Xiangtan, Hunan 411105, P. R. China\\
\noindent Email: cya@xtu.edu.cn\\

\noindent Zhenheng Li (Correspondence Author)\\
College of Mathematics and Information Science \\
Hebei University, P. R. China\\
Baoding, Hebei, China 071002\\
\noindent Email: zhli@hbu.edu.cn\\

\end{document}